\documentclass[12pt,reqno]{amsart}
\def\Z{\mathbb Z}
\def\C{\mathbb C}
\def\N{\mathbb N}

\def\1{{\bf 1}}

\def\pmod #1{\ ({\rm{mod}}\ #1)}

\usepackage{amsmath, epsfig, cite}
\usepackage{amssymb}
\usepackage{amsfonts}
\usepackage{latexsym}
\usepackage{amsthm}

\newtheorem{theorem}{Theorem}[section]

\newtheorem{corollary}[theorem]{Corollary}
\newtheorem{conjecture}[theorem]{Conjecture}
\newtheorem{lemma}[theorem]{Lemma}

\theoremstyle{remark}

\numberwithin{equation}{section}

\allowdisplaybreaks[1]

\begin{document}

\title[ $q$-Supercongruences from transformation formulas ]
{$q$-Supercongruences from transformation formulas}

\begin{abstract}
 Let $\Phi_{n}(q)$ denote the $n$-th cyclotomic polynomial in $q$. Recently, Guo and Schlosser [Constr. Approx.  53 (2021), 155--200] put forward the following conjecture: for an odd integer $n>1$,
\begin{align*}
&\sum_{k=0}^{n-1}[8k-1]\frac{(q^{-1};q^4)_k^6(q^2;q^2)_{2k}}{(q^4;q^4)_k^6(q^{-1};q^2)_{2k}}q^{8k}\notag\\
&\quad\equiv\begin{cases}0 \pmod{[n]\Phi_n(q)^2}, &\text{if }n\equiv 1\pmod{4},\\[5pt]
0 \pmod{[n]},&\text{if }n\equiv 3\pmod{4}.
\end{cases}
\end{align*}
Applying the `creative microscoping' method and several summation and transformation formulas for basic hypergeometric series and the Chinese remainder theorem for coprime polynomials,
we confirm the above conjecture, as well as another similar $q$-supercongruence conjectured by Guo and Schlosser.
\end{abstract}
\author[He-Xia Ni]{He-Xia Ni}
\address{Department of Applied Mathematics, Nanjing Audit University\\Nanjing 211815,
People's Republic of China}
\email{nihexia@yeah.net}

\author[Li-Yuan Wang]{Li-Yuan Wang}
\address{School of Physical and Mathematical Sciences, Nanjing Tech University , Nanjing 211816, People's Republic of China}
\email{wly@smail.nju.edu.cn}

\author[Hai-Liang Wu]{Hai-Liang Wu*}
\address {(Hai-Liang Wu) School of Science, Nanjing University of Posts and Telecommunications, Nanjing 210023, People's Republic of China}
\email{\tt whl.math@smail.nju.edu.cn}

\thanks{*Corresponding author.}
\thanks{The work is supported by the Natural Science Foundation of the Higher Education Institutions of Jiangsu Province (20KJB110023) and the National Natural Science Foundation of China (grants 12001279 and 12101321 ).  }
\maketitle

\section{Introduction}
In 1997, Van Hamme \cite{Ha96} proposed 13 conjectural congruences concerning $p$-adic analogues of Ramanujan-type series for $1/\pi$. He himself proved three of them.
For example, he \cite[(C.2)]{Ha96} proved the following supercongruence: for any odd prime $p$,
\begin{align}
\sum_{k=0}^{\frac{p-1}{2}}(4k+1)\frac{(\frac{1}{2})_k^4}{k!^4}&\equiv p\pmod{p^3},\label{DT1}
\end{align}
where $(a)_n=a(a+1)\cdots (a+n-1)$ is the Pochhammer symbol. Long \cite{Lo} further proved that \eqref{DT1} holds modulo $p^4$ for $p>3$.
Moreover, Van Hamme \cite[(D.2)]{Ha96} also conjectured the following relation: for $p\equiv1\pmod{6}$,
\begin{align}
\sum_{k=0}^{(p-1)/3}(6k+1)\frac{(1/3)_k^6}{k!^6}\equiv -p\Gamma_p(1/3)^9\pmod{p^4},
\end{align}
where $\Gamma_p(x)$ is the $p$-adic Gamma function.
In 2016, Long and Ramakrishna \cite[Theorem 2]{LR} established the following supercongruence:
\begin{align}\label{DT4}
\sum_{k=0}^{p-1}(6k+1)\frac{(\frac{1}{3})_k^6}{k!^6}\equiv\begin{cases}-p\Gamma_p(1/3)^9 \pmod{p^6}, &\text{if }p\equiv 1\pmod{6},\\[5pt]
-\frac{10 p^4}{27}\Gamma_p(1/3)^9 \pmod{p^6},&\text{if }p\equiv 5\pmod{6}.
\end{cases}
\end{align}
 We refer the reader to \cite{OZ16,Sw15} for the history of the proofs of Van Hamme's conjectural supercongruences.

 Recently, Guo and Zudilin \cite{GZ1} developed an analytical method, called `creative
microscoping', to prove many supercongruences by establishing their
$q$-analogues.  For more related results and the latest progress, see \cite{G2,G3,Guonew,GS,GS20,GS21,GZ1,LW,LP,NP,Wei,WY1,Zudilin2}.

In what follows, we assume $q$ to be fixed with $0<|q|<1$. For $a\in \C,$ the {\it $q$-shifted factorial} \cite{GR} is defined by
$$
(a;q)_n=\begin{cases}
(1-a)(1-aq)\cdots(1-aq^{n-1}), &\text{if }n\geq 1,\\[5pt]
1, &\text{if }n=0,
\end{cases}
$$
and the {\it $n$-th cyclotomic polynomial} is defined as
$$
\Phi_n(q):=\prod_{\substack{1\leq k\leq n\\ \gcd(n,k)=1}}(q-e^{2\pi i\cdot\frac{k}{n}}).
$$
For simplicity, we also frequently use the shortened notation:
$$(a_1,\ldots, a_m;q)_n=(a_1;q)_n\cdots (a_m;q)_n,$$
where $m\in \Z^+$ and $\ n\in \N\cup \infty$.
The main purpose of this paper is to prove the following two theorems.
\begin{theorem} \label{Conject1}
Let $n$ be a positive odd integer. Then, modulo $[n]\Phi_n(q)^2$,
\begin{align}\label{more5}
&\sum_{k=0}^{n-1}[8k+1]\frac{(q;q^4)_k^6(q^2;q^2)_{2k}}{(q^4;q^4)_k^6(q;q^2)_{2k}}q^{4k}\notag\\
&\quad\equiv\begin{cases}[n]\dfrac{(q^2;q^4)_{(n-1)/4}^2}{(q^4;q^4)_{(n-1)/4}^2} , &\text{if }n\equiv 1\pmod{4},\\[5pt]
0,&\text{if }n\equiv 3\pmod{4}.
\end{cases}
\end{align}
\end{theorem}

It is clear that Theorem \ref{Conject1} confirms \cite[Conjecture 12.4]{GS}.
Moreover, letting $n$ be an odd prime and letting $q\rightarrow 1$ in \eqref{more5}, we obtain the following conclusion.
\begin{corollary}
Let $p$ be an odd prime. Then
\begin{align*}
&\sum_{k=0}^{p-1}(8k+1)\frac{(\frac{1}{4})_k^5(\frac{1}{2})_k}{(\frac{3}{4})_kk!^5}
\equiv\begin{cases}\dfrac{p(\frac{1}{2})_{(p-1)/4}^2}{((p-1)/4)!^2} \pmod{p^3}, &\text{if }p\equiv 1\pmod{4},\\[10pt]
0 \pmod{p^3},&\text{if }p\equiv 3\pmod{4}.
\end{cases}
\end{align*}

\end{corollary}

\begin{theorem}\label{Conject2}
Let $n>3$ be an odd integer. Then, modulo $[n]\Phi_n(q)^2$,
\begin{align}\label{more6}
&\sum_{k=0}^{n-1}[8k-1]\frac{(q^{-1};q^4)_k^6(q^2;q^2)_{2k}}{(q^4;q^4)_k^6(q^{-1};q^2)_{2k}}q^{8k}\notag\\
&\quad\equiv\begin{cases}0, &\text{if }n\equiv 1\pmod{4},\\[5pt]
A_n\dfrac{(q^2;q^4)_{(n+1)/4}^2}{(q^4;q^4)_{(n+1)/4}^2} ,&\text{if
}n\equiv 3\pmod{4},
\end{cases}
\end{align}
where
\begin{align}\label{Conject2-1}
A_n=q^{n+1}(q^n-2)[n]\frac{(1-q^{1-n})^3-q^{-n-2}(1+q)(1-q^2)^2(1-q^{2-n})}{(1-q^{3-n})(1-q^{1-n})^2}.
\end{align}

\end{theorem}

It is easy to see that Theorem \ref{Conject1} confirms \cite[Conjecture 12.5]{GS}. Similarly as before, letting $n$ be an odd prime and letting $q\rightarrow 1$ in \eqref{more6}, we get the following result.
\begin{corollary}
Let $p>3$ be an odd prime. Then
\begin{align*}
&\sum_{k=0}^{p-1}(8k-1)\frac{(\frac{1}{2})_k(-\frac{1}{4})_k^5}{(\frac{1}{4})_kk!^5}
\equiv\begin{cases}0 \pmod{p^3}, &\text{if }p\equiv 1\pmod{4},\\[5pt]
\dfrac{5p(p-3)(\frac{1}{2})_{(p+1)/4}^2}{(7p-3)((p+1)/4)!^2} \pmod{p^3},&\text{if }p\equiv 3\pmod{4}.
\end{cases}
\end{align*}

\end{corollary}

Following Gasper and Rahman \cite{GR}, the ${}_{r+1}\phi_r$ basic hypergeometric series is defined by
\begin{equation*}
{}_{r+1}\phi_r\!\left[\begin{matrix}
a_1,a_2,\dots,a_{r+1}\\b_1,\dots,b_r
\end{matrix};q,z\right]:=\sum_{k=0}^\infty
\frac{(a_1,a_2,\dots,a_{r+1};q)_k}{(q,b_1,\dots,b_r;q)_k}
z^k.
\end{equation*}
In the proof of Theorems \ref{Conject1} and \ref{Conject2}, we will make use of Watson's ${}_8\phi_7$ transformation formula \cite[Appendix (III.18)]{GR}:
\begin{align}\label{Wasttran}
&{}_8\phi_7\bigg[\begin{matrix}a, &qa^{\frac{1}{2}},&-qa^{\frac{1}{2}},&b,&c,&d,&e,&q^{-m}\\ &a^{\frac{1}{2}},&-a^{\frac{1}{2}},&aq/b,&aq/c,&aq/d,&aq/e,&aq^{m+1}\end{matrix};\ q,\ \frac{a^2q^{m+2}}{bcde}\bigg]\notag\\
&\quad=\frac{(aq,aq/de;q)_m}{(aq/d,aq/e;q)_m}{}_4\phi_3\bigg[\begin{matrix}aq/{bc},&d,&e,&q^{-m}\\ &aq/b,&aq/c,&deq^{-m}/a\end{matrix};\ q,\ q\bigg],
\end{align}
and the $q$-Paff-Saalsch\"{u}tz formula (see \cite[Appendix (II.12)]{GR}):
\begin{align}\label{q-Paff}
&{}_3\phi_2\bigg[\begin{matrix}a,&b,&q^{-m}\\ &c,&abq^{1-m}/c\end{matrix};\ q,\ q\bigg]=\frac{(c/a,c/b;q)_m}{(c,c/ab;q)_m}.
\end{align}

 The rest of the paper is arranged as follows. We shall list some necessary lemmas in Section 2.
 Via several summation and transformation formulas for basic hypergeometric series and the Chinese remainder theorem for coprime polynomials, we shall prove a two-parameter
generalization of Theorem \ref{Conject1} in Section 3. Finally, the proof of  Theorem \ref{Conject2} will be given in Section 4 using the same method.

\section{Some Preparations}

We need the following lemma, which is due to Guo and Schlosser \cite[Lemma 2.1]{GS20}.
\begin{lemma}\label{Th1proofLemma1}
Let $m,n$ and $d$ be positive integers with $m\leq n-1$. Let $r$ be an integer satisfying $dm\equiv -r\pmod{n}$. Then, for $0\leq k\leq m,$ we have
\begin{align*}
\frac{(aq^r;q^d)_{m-k}}{(q^d/a;q^d)_{m-k}}\equiv (-a)^{m-2k}\frac{(aq^r;q^d)_k}{(q^d/a;q^d)_k}q^{m(dm-d+2r)/2+(d-r)k}\pmod{\Phi_{n}(q)}.
\end{align*}
\end{lemma}

From the above $q$-congruence, we can deduce the following result.
\begin{lemma}\label{Th1proofLemma2}
Let $n$ and $d$ be positive integers with $\gcd(n,d)=1$. Let $r$ be an integer. Then, modulo $\Phi_{n}(q)$,
\begin{align}
&\sum_{k=0}^{m}[2dk+r]\frac{(q^r;q^d)_k(aq^r;q^d)_k(q^r/a;q^d)_k(bq^r;q^d)_k}{(aq^d;q^d)_k(q^d/a;q^d)_k(q^d/b;q^d)_k(q^d;q^d)_k} \notag\\
&\quad\times\frac{(q^r/b;q^d)_{k}(q^2;q^d)_{k}}{(bq^d;q^d)_{k}(q^{d+r-2};q^d)_{k}}q^{(2d-2r-2)k}\equiv 0,\label{super3}
\end{align}
where $0\leq m\leq n-1$ and $dm\equiv -r\pmod{n}.$
\end{lemma}

\begin{proof} It is easy to see that Lemma \ref{Th1proofLemma2} is true  for $n=1$ or $r=0$. We now suppose that $n>1$ and $r\neq 0$. By Lemma \ref{Th1proofLemma1}, for $0\leq k\leq m,$ the $k$-th and $(m-k)$-th terms on the left-hand side of \eqref{super3} cancel each other modulo $\Phi_n(q)$, i.e.,
\begin{align*}
&[2d(m-k)+r]\frac{(q^r;q^d)_{m-k}(aq^r;q^d)_{m-k}(q^r/a;q^d)_{m-k}(bq^r;q^d)_{m-k}}{(aq^d;q^d)_{m-k}(q^d/a;q^d)_{m-k}(q^d/b;q^d)_{m-k}(q^d;q^d)_{m-k}} \notag\\
&\quad\times\frac{(q^r/b;q^d)_{m-k}(q^2;q^d)_{m-k}}{(bq^d;q^d)_{m-k}(q^{d+r-2};q^d)_{m-k}}q^{(2d-2r-2)(m-k)}\notag\\
&\quad\equiv-[2dk+r]\frac{(q^r;q^d)_k(aq^r;q^d)_k(q^r/a;q^d)_k(bq^r;q^d)_k}{(aq^d;q^d)_k(q^d/a;q^d)_k(q^d/b;q^d)_k(q^d;q^d)_k}\notag\\
&\quad\quad\times\frac{(q^r/b;q^d)_{k}(q^2;q^d)_{k}}{(bq^d;q^d)_{k}(q^{d+r-2};q^d)_{k}}q^{(2d-2r-2)k}\pmod{\Phi_n(q)}.
\end{align*}
This proves that the $q$-congruence \eqref{super3} is true modulo $\Phi_n(q)$.
\end{proof}

\begin{lemma}\label{Th1proofLemma3}
Let $n$ be a positive odd integer. Let $r=\pm1$. Then
\begin{align}
&\sum_{k=0}^{m}[8k+r]\frac{(q^r;q^4)_k^5(q^2;q^4)_k}{(q^4;q^4)_k^5(q^{2+r};q^d)_{k}}q^{(6-2r)k} \equiv 0\pmod{[n]},\label{super5}\\[5pt]
&\sum_{k=0}^{n-1}[8k+r]\frac{(q^r;q^4)_k^5(q^2;q^4)_k}{(q^4;q^4)_k^5(q^{2+r};q^d)_{k}}q^{(6-2r)k} \equiv 0\pmod{[n]},\label{super6}
\end{align}
where
\begin{align*}
&m=\begin{cases}(n-r)/4, &\text{if }n\equiv r\pmod{4},\\[5pt]
(3n-r)/4 ,&\text{if }n\equiv -r\pmod{4}.
\end{cases}
\end{align*}
\end{lemma}

\begin{proof}
By Lemma \ref{Th1proofLemma2}, the left-hand side of \eqref{super5} is congruent to $0$ modulo $\Phi_{n}(q)$. Since $4m\equiv-r \pmod{n}$,
the $q$-factorial $(q^r;q^4)_k$ has a factor of the form $1-q^{\alpha n}$ (it is congruent to $0$ modulo $\Phi_{n}(q)$) for $m<k\leq n-1$.
 Note that the polynomial $(q^4;q^4)_k^5(q^{2+r};q^4)_{k}$ does not contain the square of $\Phi_{n}(q)$, and so
$$\frac{(q^r;q^4)_k^5(q^2;q^4)_k}{(q^4;q^4)_k^5(q^{2+r};q^4)_{k}}\equiv 0\pmod{\Phi_{n}(q)}$$ for $m<k\leq n-1$.
Thus, the $k$-th summand in \eqref{super6} with $k$ satisfying $m<k\leq n-1$ is congruent to $0$ modulo $\Phi_{n}(q)$.
This together with \eqref{super5} modulo $\Phi_{n}(q)$ confirms the $q$-congruence \eqref{super6} modulo $\Phi_{n}(q)$.

We are now ready to prove \eqref{super5} and \eqref{super6} modulo $[n]$.
Let $\zeta\neq 1$ be a primitive root of unity of degree $s$ with $s|n$ and $s>1$. Let $c_q(k)$ be the $k$-th term on the left-hand side of \eqref{super5}, i.e.,
$$
c_q(k)=[8k+r]\frac{(q^r;q^4)_k^5(q^2;q^4)_k}{(q^4;q^4)_k^5(q^{2+r};q^d)_{k}}q^{(6-2r)k}.
$$
The $q$-congruences \eqref{super5} and \eqref{super6} modulo $\Phi_{n}(q)$ with $n\mapsto s$ indicate that
$$\sum_{k=0}^{m_1}c_\zeta(k)=\sum_{k=0}^{s-1}c_\zeta(k)=0,$$ where $4m_1\equiv -r\pmod{s}$ and $0\leq m_1 \leq s-1.$
Observing that
$$
\lim_{q\rightarrow \zeta}\frac{c_q(ls+k)}{c_q(ls)}=\frac{c_\zeta(k)}{[r]},
$$
we have
$$\sum_{k=0}^{n-1}c_\zeta(k)=\sum_{l=0}^{n/s-1}\sum_{k=0}^{s-1}c_\zeta(ls+k)=\frac{1}{[r]}\sum_{l=0}^{n/s-1}c_\zeta(ls)\sum_{k=0}^{s-1}c_\zeta(k)=0,$$
and
$$\sum_{k=0}^{m}c_\zeta(k)=\frac{1}{[r]}\sum_{l=0}^{(m-m_1)/s-1}c_\zeta(ls)\sum_{k=0}^{s-1}c_\zeta(k)+\frac{c_\zeta(m-m_1)}{[r]}\sum_{k=0}^{m_1}c_\zeta(k)=0,$$
which imply that both $\sum_{k=0}^{n-1}c_q(k)$ and $\sum_{k=0}^{m}c_q(k)$ are congruent to $0$ modulo $\Phi_s(q)$.
The proof then follows the fact that $\prod_{s\mid n,s>1}\Phi_{s}(q)=[n].$
\end{proof}

\section{Proof of Theorem \ref{Conject1}}
Now we display a parametric generalization of Theorem \ref{Conject1}.
\begin{theorem}\label{Th1}
Let $a,b $  be indeterminates. Let $n$ be a positive odd integer.  Then, modulo $\Phi_{n}(q)(1-aq^{tn})(a-q^{tn})(1-bq^{tn})(b-q^{tn})$,
\begin{align}
&\sum_{k=0}^{M}[8k+1]\frac{(aq,q/a,bq,q/b,q,q^2;q^4)_k}{(aq^4,q^4/a,bq^4,q^4/b,q^4,q^3;q^4)_k}q^{4k} \notag\\
&\quad\equiv [tn]\bigg(\frac{(1-bq^{tn})(b-q^{tn})(-1-a^2+aq^{tn})}{(a-b)(1-ab)}\cdot \frac{(bq^2,q^2/b;q^4)_{(tn-1)/4}}{(bq^4,q^4/b;q^4)_{(tn-1)/4}}\notag\\
&\quad\quad +\frac{(1-aq^{tn})(a-q^{tn})(-1-b^2+bq^{tn})}{(b-a)(1-ab)}\cdot \frac{(aq^2,q^2/a;q^4)_{(tn-1)/4}}{(aq^4,q^4/a;q^4)_{(tn-1)/4}}\bigg),\label{super11}
\end{align}
where
\begin{align*}
&M=\begin{cases}(n-1)/4, &\text{if }n\equiv 1\pmod{4},\\[5pt]
(3n-1)/4 ,&\text{if }n\equiv 3\pmod{4},
\end{cases}\ \text{and}\
t=\begin{cases}1, &\text{if }n\equiv 1\pmod{4},\\[5pt]
3 ,&\text{if }n\equiv 3\pmod{4}.
\end{cases}
\end{align*}
\end{theorem}
\begin{proof}
Firstly, letting $q\rightarrow q^4$ and taking $a=q, b=q^{1-tn}, c=q^{1+tn}, d= bq, e=q/b, f=q^{2}$ in \eqref{Wasttran}, we obtain
\begin{align*}
&\sum_{k=0}^{M}[8k+1]\frac{(q^{1-tn},q^{1+tn},bq,q/b,q^{2},q;q^4)_k}{(q^{4+tn},q^{4-tn},bq^4,q^4/b,q^{3},q^4;q^4)_k}q^{4k} \\
&\quad=[tn]\frac{(q,q^{3};q^4)_{(tn-1)/4}}{(q^4/b,bq^{4};q^4)_{(tn-1)/4}}\sum_{k=0}^{M}\frac{(bq,q/b,q^{1-tn};q^4)_kq^{4k}}{(q^{4-tn},q^{3},q^4;q^4)_k}.
\end{align*}
Furthermore, replacing $a=bq, b=q/b, m=(tn-1)/4, c=q^{3}$ in \eqref{q-Paff}, we have
\begin{align*}
&\sum_{k=0}^{M}[8k+1]\frac{(q^{1-tn},q^{1+tn},bq,q/b,q^{2},q;q^4)_k}{(q^{4+tn},q^{4-tn},bq^4,q^4/b,q^{3},q^4;q^4)_k}q^{4k}
=[tn]\frac{(bq^2,q^2/b;q^4)_{(tn-1)/4}}{(bq^4,q^4/b;q^4)_{(tn-1)/4}}.
\end{align*}
Namely,
\begin{align}\label{section21}
&\sum_{k=0}^{M}[8k+1]\frac{(aq,q/a,bq,q/b,q,q^2;q^4)_k}{(aq^4,q^4/a,bq^4,q^4/b,q^4,q^3;q^4)_k}q^{4k} \notag\\
&\quad\equiv [tn]\frac{(bq^2,q^2/b;q^4)_{(tn-1)/4}}{(bq^4,q^4/b;q^4)_{(tn-1)/4}}\pmod{(1-aq^{tn})(a-q^{tn})}.
\end{align}

Secondly, interchanging the parameters $a$ and $b$ in \eqref{section21}, we get the formula:
\begin{align*}
&\sum_{k=0}^{M}[8k+1]\frac{(aq,q/a,bq,q/b,q,q^2;q^4)_k}{(aq^4,q^4/a,bq^4,q^4/b,q^4,q^3;q^4)_k}q^{4k} \notag\\
&\quad\equiv [tn]\frac{(aq^2,q^2/a;q^4)_{(tn-1)/4}}{(aq^4,q^4/a;q^4)_{(tn-1)/4}}\pmod{(1-bq^{tn})(b-q^{tn})}.
\end{align*}

Thirdly, it is clear that the polynomial $(1-aq^{tn})(a-q^{tn})$ is coprime with the polynomial $(1-bq^{tn})(b-q^{tn})$. Noticing the relations:
\begin{align}
&\frac{(1-bq^{tn})(b-q^{tn})(-1-a^2+aq^{tn})}{(a-b)(1-ab)}\equiv1\pmod{(1-aq^{tn})(a-q^{tn})},\label{section222}\\
&\frac{(1-aq^{tn})(a-q^{tn})(-1-b^2+bq^{tn})}{(b-a)(1-ba)}\equiv1\pmod{(1-bq^{tn})(b-q^{tn})},\label{section223}
\end{align}
and applying the Chinese remainder theorem for coprime polynomials, we are led to the following $q$-congruence: modulo $(1-aq^{tn})(a-q^{tn})(1-bq^{tn})(b-q^{tn})$,
\begin{align*}
&\sum_{k=0}^{M}[8k+1]\frac{(aq,q/a,bq,q/b,q,q^2;q^4)_k}{(aq^4,q^4/a,bq^4,q^4/b,q^4,q^3;q^4)_k}q^{4k} \notag\\
&\quad\equiv[tn]\bigg(\frac{(1-bq^{tn})(b-q^{tn})(-1-a^2+aq^{tn})}{(a-b)(1-ab)}\cdot\frac{(bq^2,q^2/b;q^4)_{(tn-1)/4}}{(bq^4,q^4/b;q^4)_{(tn-1)/4}}\notag\\
&\quad\quad+\frac{(1-aq^{tn})(a-q^{tn})(-1-b^2+bq^{tn})}{(b-a)(1-ba)}\cdot\frac{(aq^2,q^2/a;q^4)_{(tn-1)/4}}{(aq^4,q^4/a;q^4)_{(tn-1)/4}}\bigg).
\end{align*}

Finally, by Lemma \ref{Th1proofLemma2}, the left-hand side of \eqref{super11} is congruent to $0$ modulo $\Phi_{n}(q)$. Moreover, $[tn]$ is also congruent to $0$ modulo $\Phi_n(q)$, and therefore \eqref{super11} also holds modulo $\Phi_n(q)$. Since $\Phi_n(q)$ and $(1-aq^{tn})(a-q^{tn})(b-q^{tn})(1-bq^{tn})$ are relatively prime polynomials, we can prove the $q$-supercongruence \eqref{super11} is true modulo $\Phi_{n}(q)(1-aq^{tn})(a-q^{tn})(1-bq^{tn})(b-q^{tn})$.
 This completes the proof of the theorem.
\end{proof}

Now we can prove Theorem \ref{Conject1}.
\begin{proof}[Proof of Theorems \ref{Conject1}]
Since the denominator of the reduced form of the $k$-th summand
$$\frac{(q,q,bq,q/b,q,q^2;q^4)_k}{(q^4,q^4,bq^4,q^4/b,q^4,q^3;q^4)_k}$$
does not contain the square of $\Phi_{n}(q)$, the $a\rightarrow 1$ case of Theorem \ref{Th1} reduces to
\begin{align}\label{section24}
&\sum_{k=0}^{M}[8k+1]\frac{(q,q,bq,q/b,q,q^2;q^4)_k}{(q^4,q^4,bq^4,q^4/b,q^4,q^3;q^4)_k}q^{4k} \notag\\
&\quad\equiv\mu(b,tn)\pmod{\Phi_n(q)^2(1-bq^{tn})(b-q^{tn})},
\end{align}
where
\begin{align*}
\mu(b,tn)&=[tn]\bigg(\frac{(1-bq^{tn})(b-q^{tn})(-2+q^{tn})}{(1-b)^2}\cdot\frac{(bq^2,q^2/b;q^4)_{(tn-1)/4}}{(bq^4,q^4/b;q^4)_{(tn-1)/4}}\notag\\
&\quad-\frac{(1-q^{tn})^2(-1-b^2+bq^{tn})}{(b-1)^2}\cdot\frac{(q^2;q^4)_{(tn-1)/4}^2}{(q^4;q^4)_{(tn-1)/4}^2}\bigg).
\end{align*}
Letting $b\to 1$ in \eqref{section24} and applying the L'H\^{o}spital rule, we get
\begin{align}
&\sum_{k=0}^{M}[8k+1]\frac{(q;q^4)_k^6(q^2;q^2)_{2k}}{(q^4;q^4)_k^6(q;q^2)_{2k}}q^{4k}\notag\\
&\quad=\sum_{k=0}^{M}[8k+1]\frac{(q;q^4)_k^5(q^2;q^4)_{k}}{(q^4;q^4)_k^5(q^3;q^4)_{k}}q^{4k}\notag\\
&\quad\equiv [tn]\frac{(q^2;q^4)_{(tn-1)/4}^2}{(q^4;q^4)_{(tn-1)/4}^2}\Bigg[1+(1-q^{tn})^2(q^{tn}-2)\notag\\
&\quad\quad\times\Bigg(\sum_{j=1}^{(tn-1)/4}\frac{q^{4j}}{(1-q^{4j})^2}-\sum_{j=0}^{(tn-1)/4-1}\frac{q^{4j+2}}{(1-q^{4j+2})^2}\Bigg)\Bigg]\pmod{\Phi_{n}(q)^3}.  \label{eq:added}
\end{align}

For $n\equiv 1\pmod{4}$, we have $t=1$ and $M=(n-1)/4$. From the above $q$-supercongruence we can deduce that
\begin{align}\label{section27}
&\sum_{k=0}^{(n-1)/4}[8k+1]\frac{(q;q^4)_k^6(q^2;q^2)_{2k}}{(q^4;q^4)_k^6(q;q^2)_{2k}}q^{4k}
\equiv [n]\frac{(q^2;q^4)_{(n-1)/4}^2}{(q^4;q^4)_{(n-1)/4}^2}\pmod{\Phi_{n}(q)^3},
\end{align}
where we have used the property
$$
(1-q^{n})^2\Bigg(\sum_{j=1}^{(n-1)/4}\frac{q^{4j}}{(1-q^{4j})^2}-\sum_{j=0}^{(n-1)/4-1}\frac{q^{4j+2}}{(1-q^{4j+2})^2}\Bigg)\equiv0\pmod{\Phi_{n}(q)^2}.
$$
Moreover, since $$\frac{(q;q^4)_k^6(q^2;q^2)_{2k}}{(q^4;q^4)_k^6(q;q^2)_{2k}}q^{4k}=\frac{(q;q^4)_k^5(q^2;q^4)_{k}}{(q^4;q^4)_k^5(q^3;q^4)_{k}}q^{4k}\equiv0\pmod{\Phi_n(q)^3}$$
for $k$ in the range $(n-1)/4\leq k\leq n-1$, we see that \eqref{section27} can also be written as
\begin{align}\label{section28}
&\sum_{k=0}^{n-1}[8k+1]\frac{(q;q^4)_k^6(q^2;q^2)_{2k}}{(q^4;q^4)_k^6(q;q^2)_{2k}}q^{4k}
\equiv [n]\frac{(q^2;q^4)_{(n-1)/4}^2}{(q^4;q^4)_{(n-1)/4}^2}\pmod{\Phi_{n}(q)^3}.
\end{align}

For $n\equiv 3\pmod{4}$, we have $t=3$ and $M=(3n-1)/4$. Since
the denominator of the reduced form of
$$
(1-q^{3n})^2\Bigg(\sum_{j=1}^{(3n-1)/4}\frac{q^{4j}}{(1-q^{4j})^2}-\sum_{j=0}^{(3n-1)/4-1}\frac{q^{4j+2}}{(1-q^{4j+2})^2}\Bigg)
$$ is coprime with $\Phi_{n}(q)$, and
$$
[3n]\frac{(q^2;q^4)_{(3n-1)/4}^2}{(q^4;q^4)_{(3n-1)/4}^2}\equiv0\pmod{\Phi_{n}(q)^3},
$$  we deduce from \eqref{eq:added} that
\begin{align*}
&\sum_{k=0}^{(3n-1)/4}[8k+1]\frac{(q;q^4)_k^6(q^2;q^2)_{2k}}{(q^4;q^4)_k^6(q;q^2)_{2k}}q^{4k}
\equiv 0\pmod{\Phi_{n}(q)^3}.
\end{align*}

Moreover, since
$$
\frac{(q;q^4)_k^6(q^2;q^2)_{2k}}{(q^4;q^4)_k^6(q;q^2)_{2k}}q^{4k}=\frac{(q;q^4)_k^5(q^2;q^4)_{k}}{(q^4;q^4)_k^5(q^3;q^4)_{k}}q^{4k}\equiv0\pmod{\Phi_n(q)^3}
$$
for  $(3n-1)/4\leq k\leq n-1$, we conclude that
\begin{align}\label{section30}
&\sum_{k=0}^{n-1}[8k+1]\frac{(q;q^4)_k^6(q^2;q^2)_{2k}}{(q^4;q^4)_k^6(q;q^2)_{2k}}q^{4k}\equiv 0\pmod{\Phi_{n}(q)^3}.
\end{align}
Finally, combining \eqref{section28}, \eqref{section30}, \eqref{super6} and ${\rm lcm}(\Phi_{n}(q)^3,[n])=[n]\Phi_{n}(q)^2$, we immediately obtain \eqref{more5}.
\end{proof}

\section{Proof of Theorem \ref{Conject2}}
We first give a parameter generalization of Theorem \ref{Conject2}.
\begin{theorem}\label{Th2section4}
Let $n>3$ be an odd integer. Let $a,b $  be indeterminates. Then, modulo $\Phi_{n}(q)(1-aq^{tn})(a-q^{tn})(1-bq^{tn})(b-q^{tn})$,
\begin{align}
&\sum_{k=0}^{M}[8k-1]\frac{(q^{-1},aq^{-1},q^{-1}/a,bq^{-1},q^{-1}/b,q^2;q^4)_k}{(aq^4,q^4/a,bq^4,q^4/b,q^4,q;q^4)_k}q^{8k} \notag\\
&\quad\equiv q^{-1-tn}[tn][tn+2]\notag\\
&\quad\quad\times\bigg(\frac{(1-bq^{tn})(b-q^{tn})(-1-a^2+aq^{tn})}{(a-b)(1-ab)}\frac{(bq^2,q^2/b;q^4)_{(tn+1)/4}}{(bq^4,q^4/b;q^4)_{(tn+1)/4}} T(tn,b,q)\notag\\
&\quad\quad +\frac{(1-aq^{tn})(a-q^{tn})(-1-b^2+bq^{tn})}{(b-a)(1-ab)}\frac{(aq^2,q^2/a;q^4)_{(tn+1)/4}}{(aq^4,q^4/a;q^4)_{(tn+1)/4}}T(tn,a,q)\bigg),\label{section41}
\end{align}
where
\begin{align*}
&M=\begin{cases}(3n+1)/4, &\text{if }n\equiv 1\pmod{4},\\[5pt]
(n+1)/4 ,&\text{if }n\equiv 3\pmod{4},
\end{cases}\quad
t=\begin{cases}3, &\text{if }n\equiv 1\pmod{4},\\[5pt]
1 ,&\text{if }n\equiv 3\pmod{4},
\end{cases}
\end{align*}
 and
\begin{align} \label{section412}
T(tn,b,q)&=\frac{(1-q)(q^{-2}-q^{-1-tn})}{(1-q^{-2-tn})(q^{-2}-q^{1-tn})}\notag\\
&\quad+\frac{(q^{-tn}-q^{-2-tn})(q^{-2}-q^{-tn})(bq^{-1}-q)(q^{-1}/b-q)}{(1-q^{-2-tn})(bq^{-1}-q^{-tn})(q^{-1}/b-q^{-tn})(q^{-2}-q^{1-tn})}.
\end{align}
\end{theorem}
\begin{proof}
Firstly, we shall prove the following result: modulo $(1-aq^{tn})(a-q^{tn})$,
\begin{align}\label{section413}
&\sum_{k=0}^{M}[8k-1]\frac{(q^{-1},aq^{-1},q^{-1}/a,bq^{-1},q^{-1}/b,q^2;q^4)_k}{(aq^4,q^4/a,bq^4,q^4/b,q^4,q;q^4)_k}q^{8k} \notag\\
&\quad\equiv q^{-1-tn}[tn][tn+2]T(tn,b,q)\frac{(bq^2,q^2/b;q^4)_{(tn+1)/4}}{(bq^4,q^4/b;q^4)_{(tn+1)/4}}.
\end{align}
For $a=q^{tn}$ or $a=q^{-tn}$, by the Watson's ${}_8\phi_7$ transformation formula \eqref{Wasttran},
the left-hand side of \eqref{section413} is equal to
\begin{align}\label{section42}
&\sum_{k=0}^{M}[8k-1]\frac{(q^{-1},q^{-1-tn},q^{-1+tn},bq^{-1},q^{-1}/b,q^2;q^4)_k}{(q^{4+tn},q^{4-tn},bq^4,q^4/b,q^4,q;q^4)_k}q^{8k} \notag\\
&\quad=[-1]\cdot{}_8\phi_7\bigg[\begin{matrix}q^{-1}, &q^{\frac{7}{2}},&-q^{\frac{7}{2}},&bq^{-1},&q^{-1}/b,&q^{-1+tn},&q^{-1-tn},&q^{2}\\ &q^{-\frac{1}{2}},&-q^{-\frac{1}{2}},&q^{4-tn},&q^{4+tn},&bq^4,&q^4/b,&q\end{matrix};\ q^4,\ q^{8}\bigg]\notag\\
&\quad=[tn]\frac{(q^{-1},q^5;q^4)_{(tn+1)/4}}{(bq^4,q^4/b;q^4)_{(tn+1)/4}}{}_4\phi_3\bigg[\begin{matrix}q^{2-tn}, &bq^{-1},&q^{-1}/b,&q^{-1-tn}\\ &q^{4-tn},&q,&q^{-2-tn}\end{matrix};\ q^4,\ q^{4}\bigg].
\end{align}
Moreover, letting $q\rightarrow q^4, a\rightarrow bq^{-1}, b\rightarrow q^{-1}/b, c\rightarrow q^{-tn}, x\rightarrow q^{-2-tn} ,m=(tn+1)/4$ in the formula (cf. \cite[(2.5)]{GC21}),
we see that the $_4\phi_3$ summation on the right-hand side of \eqref{section42} is equal to
\begin{align*}
q^{-tn-1}T(tn,b,q)\frac{(bq^2,q^2/b;q^4)_{(tn+1)/4}}{(q^{-1},q;q^4)_{(tn+1)/4}}.
\end{align*}
Namely, the identity \eqref{section42} may be simplified as
\begin{align*}
&\sum_{k=0}^{M}[8k-1]\frac{(q^{-1},q^{-1-tn},q^{-1+tn},bq^{-1},q^{-1}/b,q^2;q^4)_k}{(q^{4+tn},q^{4-tn},bq^4,q^4/b,q^4,q;q^4)_k}q^{8k} \notag\\
&\quad=q^{-tn-1}[tn][tn+2]T(tn,b,q)\frac{(bq^2,q^2/b;q^4)_{(tn+1)/4}}{(bq^4,q^4/b;q^4)_{(tn+1)/4}}.
\end{align*}
This proves that the $q$-congruence \eqref{section413} holds.

Secondly, interchanging the parameters $a$ and $b$ in \eqref{section413}, we get the $q$-congruence:
modulo $(1-bq^{tn})(b-q^{tn})$,
\begin{align*}
&\sum_{k=0}^{M}[8k-1]\frac{(q^{-1},aq^{-1},q^{-1}/a,bq^{-1},q^{-1}/b,q^2;q^4)_k}{(aq^4,q^4/a,bq^4,q^4/b,q^4,q;q^4)_k}q^{8k} \notag\\
&\quad\equiv q^{-tn-1}[tn][tn+2]T(tn,a,q)\frac{(aq^2,q^2/a;q^4)_{(tn+1)/4}}{(aq^4,q^4/a;q^4)_{(tn+1)/4}}.
\end{align*}

Finally, similarly as before, employing Lemma \ref{Th1proofLemma2}, the $q$-congruences \eqref{section222}, \eqref{section223} and the Chinese remainder theorem for coprime polynomials, we obtain the $q$-congruence \eqref{section41}.
\end{proof}
\begin{proof}[Proof of Theorem \ref{Conject2} ]
As we have already mentioned in the proof of Theorem \ref{Conject1}, the denominator of the reduced form of the $k$-th summand
$$\frac{(q^{-1},q^{-1},q^{-1},bq^{-1},q^{-1}/b,q^2;q^4)_k}{(q^4,q^4,bq^4,q^4/b,q^4,q;q^4)_k}q^{8k}$$
does not contain the square of $\Phi_{n}(q)$. Letting $a\rightarrow1$ in Theorem \ref{Th2section4}, we conclude that, modulo $\Phi_{n}(q)^2(1-bq^{tn})(b-q^{tn})$,
\begin{align*}
&\sum_{k=0}^{M}[8k-1]\frac{(q^{-1},q^{-1},q^{-1},bq^{-1},q^{-1}/b,q^2;q^4)_k}{(q^{4},q^{4},bq^4,q^4/b,q^4,q;q^4)_k}q^{8k} \notag\\
&\quad\equiv q^{-1-tn}[tn][tn+2]\nu(tn,b,q),
\end{align*}
where
\begin{align*}
\nu(tn,b,q)&=\frac{(1-bq^{tn})(b-q^{tn})(-2+q^{tn})}{(1-b)^2}\cdot \frac{(bq^2,q^2/b;q^4)_{(tn+1)/4}}{(bq^4,q^4/b;q^4)_{(tn+1)/4}}T(tn,b,q)\notag\\
&\quad -\frac{(1-q^{tn})^2(-1-b^2+bq^{tn})}{(1-b)^2}\cdot \frac{(q^2;q^4)_{(tn+1)/4}^2}{(q^4;q^4)_{(tn+1)/4}^2}T(tn,1,q).
\end{align*}
By the L'H\^{o}spital rule, we are led to
\begin{align*}
&\lim_{b\rightarrow 1}\nu(tn,b,q)\\
&\quad=\frac{(q^2;q^4)_{(tn+1)/4}^2}{(q^4;q^4)_{(tn+1)/4}^2}\Bigg[(1-q^{tn})^2(q^{tn}-2)T(tn,1,q)\\
&\quad\quad\times\Bigg(\sum_{j=1}^{(tn+1)/4}\frac{q^{4j}}{(1-q^{4j})^2}-\sum_{j=0}^{(tn+1)/4-1}\frac{q^{4j+2}}{(1-q^{4j+2})^2}\Bigg)\\
&\quad\quad-q^{tn}(q^{tn}-2)T(tn,1,q)+(1-q^{tn})^2(q^{tn}-2)\cdot\frac{\partial T(tn,b,q)}{\partial b}\bigg|_{b=1}\\
&\quad+\frac{1}{2}(1-q^{tn})^2(q^{tn}-2)\cdot\frac{\partial^2 T(tn,b,q)}{\partial b^2}\bigg|_{b=1}+(1-q^{tn})^2T(tn,1,q)\Bigg],
\end{align*}
where $T(tn,b,q)$ is given by \eqref{section412}.

Furthermore, it is easy to check that both the denominators of $\frac{\partial T(tn,b,q)}{\partial b}\bigg|_{b=1}$ and $\frac{\partial^2 T(tn,b,q)}{\partial b^2}\bigg|_{b=1}$ are not divisible by $\Phi_{n}(q)$. Hence, modulo $\Phi_{n}(q)^3$, we have
\begin{align*}
&\sum_{k=0}^{M}[8k-1]\frac{(q^{-1};q^4)_k^6(q^2;q^2)_{2k}}{(q^4;q^4)_k^6(q^{-1};q^2)_{2k}}q^{8k}\notag\\
&\quad=\sum_{k=0}^{M}[8k-1]\frac{(q^{-1};q^4)_k^5(q^2;q^4)_{k}}{(q^4;q^4)_k^5(q;q^4)_{k}}q^{8k}\notag\\
&\quad\equiv q^{-1-tn}[tn][tn+2]\frac{(q^2;q^4)_{(tn+1)/4}^2}{(q^4;q^4)_{(tn+1)/4}^2}\Bigg[(1-q^{tn})^2(q^{tn}-2)T(tn,1,q)\\
&\quad\quad\times\Bigg(\sum_{j=1}^{(tn+1)/4}\frac{q^{4j}}{(1-q^{4j})^2}-\sum_{j=0}^{(tn+1)/4-1}\frac{q^{4j+2}}{(1-q^{4j+2})^2}\Bigg)-q^{tn}(q^{tn}-2)T(tn,1,q)\Bigg].
\end{align*}

For $n\equiv1\pmod {4}$, we have $t=3$ and $M=(3n+1)/4$. Since the expression $(q^2;q^4)_{(3n+1)/4}^2/(q^4;q^4)_{(3n+1)/4}^2$ is congruent to $0$ modulo $\Phi_{n}(q)^2$,
and the denominator
of the reduced form of the fraction
\begin{align*}
&(1-q^{3n})^2(q^{3n}-2)T(3n,1,q)\Bigg(\sum_{j=1}^{(3n+1)/4}\frac{q^{4j}}{(1-q^{4j})^2}-\sum_{j=0}^{(3n+1)/4-1}\frac{q^{4j+2}}{(1-q^{4j+2})^2}\Bigg)
\end{align*}
 is relatively prime to $\Phi_{n}(q)$,
we immediately get
\begin{align*}
&\sum_{k=0}^{(3n+1)/4}[8k-1]\frac{(q^{-1};q^4)_k^6(q^2;q^2)_{2k}}{(q^4;q^4)_k^6(q^{-1};q^2)_{2k}}q^{8k}
\equiv0\pmod{\Phi_{n}(q)^3}.
\end{align*}
Moreover, noticing that $$\frac{(q^{-1};q^4)_k^6(q^2;q^2)_{2k}}{(q^4;q^4)_k^6(q^{-1};q^2)_{2k}}\equiv0\pmod{\Phi_n(q)^3}$$
for $(3n+1)/4\leq k\leq n-1$, we conclude that
\begin{align}
&\sum_{k=0}^{n-1}[8k-1]\frac{(q^{-1};q^4)_k^6(q^2;q^2)_{2k}}{(q^4;q^4)_k^6(q^{-1};q^2)_{2k}}q^{8k}
\equiv 0\pmod{\Phi_{n}(q)^3}.\label{section500}
\end{align}

For $n\equiv 3\pmod {4}$, we have $t=1$ and $M=(n+1)/4$. Similarly as before, since $[n]\equiv0\pmod{\Phi_n(q)}$ for $n>1$ and $(q^2;q^4)_{(n+1)/4}$ is not divisible by $\Phi_{n}(q)$, we obtain
\begin{align*}
&\sum_{k=0}^{(n+1)/4}[8k-1]\frac{(q^{-1};q^4)_k^6(q^2;q^2)_{2k}}{(q^4;q^4)_k^6(q^{-1};q^2)_{2k}}q^{8k}\notag\\
&\quad\equiv -q^{-1}(q^{n}-2)[n][n+2]T(n,1,q)\frac{(q^2;q^4)_{(n+1)/4}^2}{(q^4;q^4)_{(n+1)/4}^2}\notag\\
&\quad=A_n\frac{(q^2;q^4)_{(n+1)/4}^2}{(q^4;q^4)_{(n+1)/4}^2}\pmod{\Phi_{n}(q)^3},
\end{align*}
where $A_n$ is given by \eqref{Conject2-1}.

On the other hand, one sees that the following $q$-congruence  holds
\begin{align}
&\sum_{k=0}^{n-1}[8k-1]\frac{(q^{-1};q^4)_k^6(q^2;q^2)_{2k}}{(q^4;q^4)_k^6(q^{-1};q^2)_{2k}}q^{8k}\notag\\
&\quad\equiv A_n\frac{(q^2;q^4)_{(n+1)/4}^2}{(q^4;q^4)_{(n+1)/4}^2}\pmod{\Phi_{n}(q)^3}.
\label{section49}
\end{align}
by noticing that $(q^{-1};q^4)_k^6(q^2;q^2)_{2k}/(q^4;q^4)_k^6(q^{-1};q^2)_{2k}\equiv0\pmod{\Phi_{n}(q)^3}$ for $(n+1)/4\leq k\leq n-1$.

Finally, combining \eqref{section500}, \eqref{section49},  Lemma \ref{Th1proofLemma3} and ${\rm lcm}(\Phi_{n}(q)^3,[n])=[n]\Phi_{n}(q)^2$, Theorem \ref{Conject2} is concluded.

\end{proof}

\section{An open problem}
Guo and Zudilin \cite[Theorem 2]{GZ2} established the following $q$-supercongruence:
modulo $\Phi_n(q)^2$,
\begin{align}
\sum_{k=0}^{n-1}\frac{(q;q^2)_k^2(q^2;q^4)_k}{(q^2;q^2)_k^2(q^4;q^4)_k}q^{2k}
\equiv\begin{cases}
\dfrac{(q^2;q^4)_{(n-1)/4}^2}{(q^4;q^4)_{(n-1)/4}^2}q^{(n-1)/2} &\text{if}\; n\equiv1\pmod4, \\[2.5pt]
0 &\text{if}\; n\equiv3\pmod4,
\end{cases}
\label{eq:mod-phi}
\end{align}
which is a $q$-analogue of the (H.2) supercongruence of Van Hamme \cite{Ha96}.
Combining \eqref{more5} and \eqref{eq:mod-phi}, we have
\begin{align}\label{more-fin}
&\sum_{k=0}^{n-1}[8k+1]\frac{(q;q^4)_k^6(q^2;q^2)_{2k}}{(q^4;q^4)_k^6(q;q^2)_{2k}}q^{4k}
\equiv [n]q^{(1-n)/2}\sum_{k=0}^{n-1}\frac{(q;q^2)_k^2(q^2;q^4)_k}{(q^2;q^2)_k^2(q^4;q^4)_k}q^{2k}\pmod{[n]\Phi_n(q)^2}
\end{align}
for odd $n$. Letting $n=p^r$ be an odd prime power and then taking $q\to 1$ in \eqref{more-fin}, we get
\begin{align*}
&\sum_{k=0}^{p^r-1}(8k+1)\frac{(\frac{1}{4})_k^5(\frac{1}{2})_k}{(\frac{3}{4})_kk!^5}
\equiv p^r\sum_{k=0}^{p^r-1}\frac{(\frac{1}{2})_k^3}{k!^3} \pmod{p^{r+2}}.
\end{align*}

It seems that the above supercongruence can be strengthened as follows.
\begin{conjecture}Let $p$ be an odd prime and let $r\geqslant 1$. Then
\begin{align*}
&\sum_{k=0}^{p^r-1}(8k+1)\frac{(\frac{1}{4})_k^5(\frac{1}{2})_k}{(\frac{3}{4})_kk!^5}
\equiv p^r\sum_{k=0}^{p^r-1}\frac{(\frac{1}{2})_k^3}{k!^3}
\begin{cases}\pmod{p^{r+3}}, &\text{if $p\equiv 1\pmod{4}$},\\
\pmod{p^{2r+1}}, &\text{if $p\equiv 3\pmod{4}$}.
\end{cases}
\end{align*}
\end{conjecture}


\end{document}